\documentclass[12pt]{amsart}
\usepackage[all]{xy}
\usepackage{amssymb}
\usepackage{amsthm}
\usepackage{amsmath}
\usepackage{amscd,enumitem}
\usepackage{verbatim}
\usepackage{eurosym}
\usepackage{graphicx}
\usepackage{dsfont}
\usepackage{stmaryrd}
\usepackage{color}
\usepackage{float}
\usepackage{longtable}
\usepackage{dcolumn}
\usepackage[mathscr]{eucal}
\usepackage[all]{xy}
\usepackage{hyperref}
\usepackage[title]{appendix}
\usepackage[usenames,dvipsnames]{xcolor}
\usepackage{bbm}
\usepackage[textheight=8.85in, textwidth=6.8in]{geometry}
\usepackage{multirow}
\usepackage{caption}
\theoremstyle{plain}
\newtheorem{theorem}{Theorem}[section]
\newtheorem{corollary}[theorem]{Corollary}
\newtheorem{lemma}[theorem]{Lemma}
\newtheorem{proposition}[theorem]{Proposition}

\theoremstyle{definition}

\theoremstyle{remark}
\newtheorem*{remark}{Remark}
\newtheorem*{example}{Example}

\newcommand{\Z}{\mathbb{Z}}
\newcommand{\Q}{\mathbb{Q}}
\newcommand{\GL}{\mathrm{GL}}
\newcommand{\ch}{\mathrm{char}}
\newcommand{\F}{\mathbb{F}}
\newcommand{\p}{\mathbb{P}}
\newcommand{\im}{\mathrm{Im}}
\newcommand{\re}{\mathrm{Re}}

\newcommand{\onehalf}{\frac{1}{2}}
\newcommand{\R}{\mathbb{R}}
\newcommand{\HH}{\mathbb{H}}

\newcommand{\N}{\mathbb{N}}
\newcommand{\SL}{\operatorname{SL}}
\newcommand{\Cl}{\mathrm{Cl}}
\newcommand{\Leg}{\mathrm{Leg}}
\newcommand{\C}{\mathbb{C}}

\newcommand{\leg}[2]{\genfrac{(}{)}{}{}{#1}{#2}}

\catcode`,\active

\catcode`\,12

\numberwithin{equation}{section}

\begin{document}
\title[Distribution of values of Gaussian hypergeometric functions]{Distribution of values of Gaussian hypergeometric functions}

\dedicatory{In celebration of Don Zagier's 70th birthday}
\author{Ken Ono, Hasan Saad, \and Neelam Saikia}

\address{Department of Mathematics, University of Virginia, Charlottesville, VA 22904}
 \email{ken.ono691@virginia.edu}
 \email{hs7gy@virginia.edu}
 \email{nlmsaikia1@gmail.com}

\keywords{Gaussian hypergeometric functions, Distributions, Elliptic curves}
\subjclass[2000]{11F46, 11F11, 11G20, 11T24, 33E50}

\begin{abstract}  In the 1980's, Greene defined {\it hypergeometric functions over finite fields} using Jacobi sums. The framework of his theory establishes that these functions possess many properties that are analogous to those of the classical hypergeometric series studied by Gauss and Kummer. These functions have played important roles in the study of Ap\'ery-style supercongruences, the Eichler-Selberg trace formula, Galois representations, and zeta-functions of arithmetic varieties. We study  the value distribution (over large finite fields) of natural families of these functions. For the $_2F_1$ functions, the limiting distribution is  semicircular (i.e. $SU(2)$), whereas the distribution for the $_3F_2$ functions is the  {\it Batman} distribution for the traces of the real orthogonal group $O_3$.
\end{abstract}
\maketitle
\section{Introduction and statement of results}

 In the '80s, Greene  \cite{GreenePhD, greene} defined  {\it Gaussian hypergeometric functions} over finite fields
using Jacobi sums.
He developed the foundation of a beautiful theory where these functions possess many properties
that are analogous to those of classical hypergeometric functions. These properties include transformation laws, explicit evaluations, and contiguous relations. These functions have played central roles in the study of combinatorial supercongruences
\cite{Ahlgren, AO, KC, McCarthyRV, McCarthyOsburn, MOS, Mortenson, OsburnSchneider, OSS, OSZ,  OZ, Papanikolas},
Dwork hypersurfaces \cite{BRS, McCarthyRIMS}, Galois representations \cite{LLL, Long}, $L$-functions of elliptic curves \cite{BK, BS, BarmanTripathi, Fuselier, Lennon, McCarthyPJM, ono, Rouse, TY},  hyperelliptic curves \cite{BKS, BSM-JNT}, $K3$ surfaces \cite{ono-3, DKSSVW, ono}, Calabi-Yau threefolds \cite{ahlgren-ono, AO, Z},  the Eichler-Selberg trace formula \cite{fop, Fuselier, FuselierHecke, FM, Lennon1, MP, Papanikolas, PujahariSaikia}, among other topics. This body of work meshes well with the framework established by Katz \cite{Katz} and Roberts and Villegas \cite{RV} on the analysis and arithmetic of ``hypergeometric varieties''.

Here we initiate the study of the value distribution of Greene's functions.
We first recall  Greene's original definition. If $A_1,A_2,\ldots,A_n$ and $B_1,B_2,\ldots,B_{n-1}$ are multiplicative characters of the finite field $\F_q$, where $q=p^r$, then we have the Gaussian hypergeometric function
\begin{align}
{_{n}F_{n-1}}\left(\begin{array}{cccc}
A_1, & A_2,& \ldots, & A_n\\
~& B_1, &\ldots, &B_{n-1}
\end{array}\mid x\right)_q:=\frac{q}{q-1}\sum_{\chi}{A_1\chi\choose\chi}{A_2\chi\choose B_1\chi}\cdots{A_n\chi\choose B_{n-1}\chi}\chi(x),\notag
\end{align}
where the summation is over the multiplicative characters\footnote{For characters $\chi$, we have that $\chi(0):=0.$} of $\F_q^{\times},$ and where  ${A\choose B}$ is the normalized Jacobi sum $J(A,B)$, defined by
\begin{align}\label{binomial}
{A\choose B}:=\frac{B(-1)}{q}J(A,\overline{B}):=\frac{B(-1)}{q}\sum\limits_{x\in\F_q}A(x)\overline{B}(1-x).
\end{align}

Many authors (see  \cite{greene}, \cite{GreeneStanton},  \cite{ISS}, \cite{Koike},  \cite{ono}, and \cite{Rouse}, to name a few) have made use of the mantra that Gaussian analogs of classical hypergeometric results arise when rational parameters
$1/n$ are replaced with a character $\chi$ of order $n$ (resp. $a/n$ with $\chi^a$). We  consider those
functions where the parameter characters have order 1 and 2, which always exist for $\F_q$ when $q=p^r$ is odd.
The simplest example of these functions are the ${_2F_1}$-Gaussian hypergeometric functions 
\begin{equation}
{_2F_1}(\lambda)_q:={_2F_1}\left(\begin{array}{cc}
\phi, & \phi\\
~& \varepsilon
\end{array}\mid \lambda\right)_q=\frac{q}{q-1}\sum\limits_{\chi}{\phi\chi\choose\chi}{\phi\chi\choose\chi}\chi(\lambda),
\end{equation}
 where $\phi(\cdot)$ is the quadratic character and $\varepsilon$ is the trivial character of $\F_q$.
As our first result, we compute the moments of these Gaussian hypergeometric functions.

\begin{theorem}\label{Theorem2F1}
If $r$ and $m$ are fixed positive integers, then as $p\rightarrow +\infty$ we have
$$
p^{r(m/2-1)}\sum\limits_{\lambda\in\F_{p^r}}{_2F_1}(\lambda)_{p^r}^m=
\begin{cases} o_{m,r}(1) \ \ \ \ \ &{\text {if $m$ is odd}}\\ 
\frac{(2n)!}{n!(n+1)!}+o_{m,r}(1) \ \ \ \
&{\text {if $m=2n$ is even.}}
\end{cases}
$$
\end{theorem}
\begin{remark}The non-zero moments in Theorem~\ref{Theorem2F1} (i.e. the Catalan numbers) arise \cite{Drew} as the moments of traces of  the Lie group $SU(2),$ the $2\times2$  determinant 1 unitary matrices.
Namely, for even moments, we have
$$
\int_{SU(2)} (\text{Tr} X)^{2n} dX=\frac{(2n)!}{n!(n+1)!},
$$
where the integral is with respect to the Haar measure on $SU(2)$.

\end{remark}

Using these moments, we determine the limiting behavior of the
$_2F_1(\lambda)_{p^r}$ as $p\rightarrow +\infty.$  We obtain the limiting distribution of the renormalized values
$\sqrt{p^r}\cdot\  _2F_1(\lambda)_{p^r}\in [-2, 2],$
which we view as random variables on $\F_{p^r}.$
Namely, we obtain the following result.

\begin{corollary}\label{Distribution2F1}
If $-2\leq a<b\leq 2$, and $r$ is a fixed positive integer, then 
$$
\lim_{p\to\infty}\frac{|\left\{\lambda\in\F_{p^r}: \sqrt{p^r}\cdot{_2F_1}(\lambda)_{p^r} \in [a,b]\right\}|}{p^r}=\frac{1}{2\pi}\int_{a}^b \sqrt{4-t^2} dt.
$$
\end{corollary}

\begin{remark} 
Theorem~\ref{Theorem2F1} may be interpreted in terms of the Legendre normal form elliptic curves
$$
E_\lambda^\Leg:\ \ \ y^2=x(x-1)(x-\lambda).
$$
If $\lambda\in\F_q\setminus\{0,1\},$ then
 (see Theorem 11.10 of \cite{ono-book}) 
$q\cdot {_2F_1}(\lambda)_q=-\phi(-1)\cdot a_\lambda^{\Leg}(q),$
where 
\begin{equation}\label{TraceCharacterSum}
a_\lambda^{\Leg}(q):=q+1- |E_\lambda^\Leg(\F_q)|=-\sum_{x\in \F_q} \phi(x(x-1)(x-\lambda)).
\end{equation}
Corollary~\ref{Distribution2F1} refines (i.e. restriction to Legendre curves) a classical theorem of Birch \cite{birch} which established this distribution for all elliptic curves over finite fields. Birch's Theorem has recently been refined \cite{BKP} by Bringmann, Kane, and Pujahari  in another direction, where the Frobenius traces  are restrictied to arithmetic progressions.
These distributions are renormalizations of the usual Sato-Tate distribution which was famously proved by Clozel, Harris, Shepherd-Barron and Taylor in \cite{SatoTateProof}. In their (more difficult) setting, the elliptic curve is fixed and the distribution is taken over all primes $p.$ Recent work along these lines for further abelian varieties have been obtained by Fit\'e, Kedlaya, and Sutherland  (for example, see \cite{Drew}).
\end{remark}

We also consider these questions for the $_3F_2$ Gaussian hypergeometric functions
\begin{equation}
{_3F_2}(\lambda)_q:={_3F_2}\left(\begin{array}{ccc}
\phi, & \phi, &\phi \\
~& \varepsilon, &\varepsilon
\end{array}\mid \lambda\right)_q=\frac{q}{q-1}\sum\limits_{\chi}{\phi\chi\choose\chi}{\phi\chi\choose\chi}{\phi\chi\choose\chi}\chi(\lambda).
\end{equation}
The power moments for these functions satisfy the following asymptotics.

\begin{theorem}\label{Theorem3F2}
If $r$ and $m$ are fixed positive integers, then as $p\rightarrow +\infty$ we have
$$
p^{r(m-1)}\sum\limits_{\lambda\in\F_{p^r}} {_3F_2}(\lambda)_{p^r}^{m}=\begin{cases} o_{m,r}(1) \ \ \ \ \ &{\text { if $m$ is odd}}\\ 
\sum\limits_{i=0}^{m} (-1)^i{m\choose i}\frac{(2i)!}{i!(i+1)!}+o_{m,r}(1)  \ \ \ \ 
&{\text { if $m$ is even.}}
\end{cases}
$$
\end{theorem}
\begin{remark} The moments in Theorem~\ref{Theorem3F2} arise \cite{PV} as moments of traces of the real orthogonal group $O_3$.
Namely, for even $m$ we have
$$
\int_{O_3} (\text{Tr} X)^m dX=\sum\limits_{i=0}^{m} (-1)^i{m\choose i}\frac{(2i)!}{i!(i+1)!},
$$
where the integral is with respect to the Haar measure on $O_3$.
\end{remark}

In analogy with Corollary~\ref{Distribution2F1}, we obtain the  limiting distribution of the renormalized values
$p^r\cdot {_3F_2(\lambda)_{p^r}} \in [-3, 3],$
that we view as random variables over $\F_{p^r}.$
We obtain the following result.

\begin{corollary}\label{Distribution3F2}
 If  $-3\leq a<b\leq 3,$ and $r$ is a fixed positive integer, then 
$$
\lim_{p\to\infty}\frac{|\left\{\lambda\in\F_{p^r}: p^r\cdot{_3F_2}(\lambda)_{p^r}\in [a,b] \right\}|}{p^r}=\frac{1}{4\pi}\int_a^b f(t)dt,
$$ 
where 
$$
f(t)=\begin{cases}
\frac{3-|t|}{\sqrt{3+2|t|-t^2}} \ \ \ \ &\ \ \ \ {\text {\it if}}\  1<|t|<3, \\ 
\frac{3+t}{\sqrt{3-2t-t^2}}+\frac{3-t}{\sqrt{3+2t-t^2}} &\ \ \ \  {\text{if}}\  |t|<1,\\ 
0 &\ \ \ \ \text{otherwise}.
\end{cases}
$$
\end{corollary}

\begin{example}  For the prime $p=93283$ (i.e. $r=1$), the histograms of
 the values $\sqrt{p}\cdot {_2F_1}(\lambda)_p$ and  $p \cdot {_3F_2}(\lambda)_p$
 illustrate
 Corollary~\ref{Distribution2F1}  (i.e. the near match with the radius 2 semicircle) and
 Corollary\ref{Distribution3F2} (i.e. the near match with the Batman distribution $f(t)$).

\begin{center}
\begin{table}[H]
 \ \ \ \ \ \ \  \ \ \ \ \ \ \includegraphics[height=35mm]{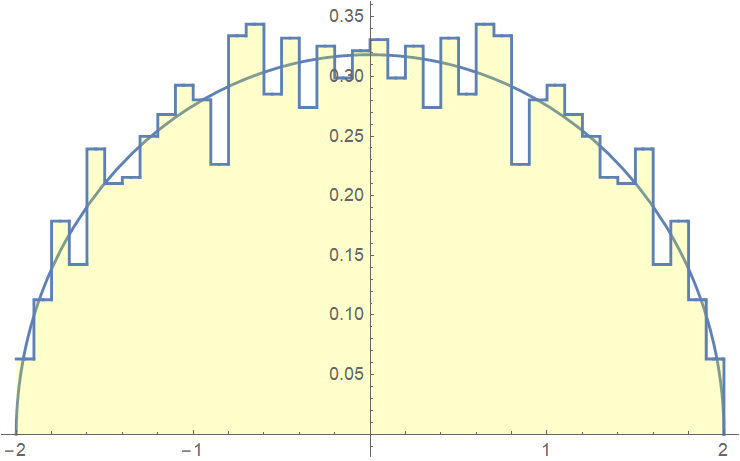}\ \   \ \ \ \ \ \ \ \  \ \ \ \ \ \ \ \ \includegraphics[height=45mm]{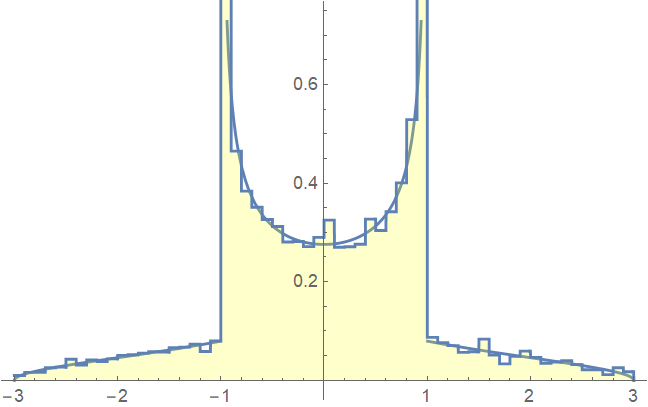}\\
\caption*{ \ \ \ $_2F_1$ histogram for $p=93283$ \ \ \ \ \ \ \ \ \ \ \ \ \ \ \ \ \ \ \ \ \ \ \  \ \ \ $_3F_2$ histogram for $p=93283$}
\end{table}
\end{center}
\end{example}

\begin{remark}
Theorem~\ref{Theorem3F2} can be interpreted in terms of the $K3$ surfaces whose function fields are 
$$
X_\lambda:\ \ \ s^2=xy(x+1)(y+1)(x+\lambda y),
$$
where $\lambda\in\F_q\setminus\{0,-1\}$. It is known (see Theorem 11.18 of \cite{ono-book} and Proposition 4.1 of \cite{ono-3}) that
 $$|X_\lambda(\F_q)|=1+q^2+19q+q^2\cdot {_3F_2}(-\lambda)_q.
 $$
  Corollary~\ref{Distribution3F2} gives the limiting Frobenius trace distribution for these $K3$ surfaces.
\end{remark}

\begin{remark} It is natural  to consider the asymptotics for the moments of the general
$$
{_{n}F_{n-1}}(\lambda)_{p^r}:={_{n}F_{n-1}}\left(\begin{array}{cccc}
\phi, & \phi,  &\dots, &\phi \\
~& \varepsilon,  &\dots, &\varepsilon
\end{array}\mid \lambda\right)_{p^r}=\frac{p^r}{p^r-1}\sum\limits_{\chi}{\phi\chi\choose\chi}{\phi\chi\choose\chi}\cdots{\phi\chi\choose\chi}\chi(\lambda)
$$
hypergeometric functions. 
It would be very interesting to determine asymptotics for the moments, which in turn would lead
 to  distributions that extend Corollaries~\ref{Distribution2F1} and \ref{Distribution3F2}.  A solution to this problem in the case of the $_4F_3$ functions is already quite interesting.
\end{remark}

The proofs of Theorems~\ref{Theorem2F1} and \ref{Theorem3F2} rely on the fact that the
$ _2F_1(\lambda)_q$ and $ _3F_2(\lambda)_q$ values arise from the arithmetic of the Legendre and Clausen elliptic curves
\begin{eqnarray}\label{FamiliesofCurves}
E^{\Leg}_{\lambda}:  \ \  y^2=x(x-1)(x-\lambda) \ \ \ {\text {\rm and}}\ \ \ 
E^{\Cl}_\lambda: \ \ y^2=(x-1)\left(x^2+\lambda\right).
\end{eqnarray}
As mentioned above, the $_2F_1(\lambda)_q$ are renormalizations of the Frobenius traces for $E^{\Leg}_{\lambda}/\F_q.$ 
The $_3F_2(\lambda)_q$ (see Theorem~\ref{Trace3F2})
are related to the squares of the Frobenius traces of $E^{\Cl}_{\lambda}/\F_q.$  
Using these arithmetic geometric connections, we  reformulate the moments in terms of the moduli space of these elliptic curves. 
We interpret these reformulations in terms of isomorphism classes of elliptic curves with certain subgroups of $\F_q$ rational points. The moments
can then be given as weighted sums of Hurwitz class numbers which enumerate such isomorphism classes.

To estimate these moments, we make use of the theory of harmonic Maass forms.
More precisely,  these weighted sums arise in the Fourier expansions of nonholomorphic modular forms produced from the Rankin-Cohen brackets of Zagier's weight 3/2 nonholomorphic Eisenstein series when paired with explicit theta functions. The proofs of Theorems~\ref{Theorem2F1} and \ref{Theorem3F2}  are then reduced to an application of Deligne's Theorem, which bounds the coefficients of  the cuspidal components of the holomorphic projections of these nonholomorphic modular forms. The recent proof of Cohen's Conjecture by Mertens \cite{mertens, mertensRMS} plays a significant role in the $_2F_1$ case.

This paper is organized as follows.
In Section~\ref{Section2F1} we recall the fundamental facts we require about the $_2F_1(\lambda)_q$ functions and the arithmetic of the Legendre  curves $E_{\lambda}^{\Leg}$. In Section~\ref{Section3F2} we recall the analogous results for $_3F_2(\lambda)_q$ and the Clausen curves $E_{\lambda}^{\Cl}$. In Section~\ref{HarmonicMaassForms}, we recall facts from the theory of harmonic Maass forms which enable us to analyze  these elliptic curves in terms of weighted class number sums thanks to a theorem of Schoof. We apply these facts in
Section~\ref{WeightedSumsOfClass} to obtain the asymptotic properties of these class number sums.
 In Section~\ref{ProbTheory} we recall the criteria for deducing the distributions in Corollaries~\ref{Distribution2F1} and \ref{Distribution3F2} in terms of moments. Finally, in Section~\ref{Proofs} we conclude with the proofs of 
Theorems~\ref{Theorem2F1} and \ref{Theorem3F2}. 
\section*{Acknowledgements} \noindent
The authors thank Quanlin Chen, Eric Shen, and Drew Sutherland for comments on earlier drafts of this article. The first
  author thanks  the Thomas Jefferson Fund and the NSF
(DMS-2002265 and DMS-2055118) for their generous support, as well as  the Kavli Institute grant NSF PHY-1748958.
The third author is grateful for the support of a Fulbright Nehru Postdoctoral Fellowship.

\section{The $_2F_1(\lambda)_q$ and the arithmetic of $E^{\Leg}_\lambda$}\label{Section2F1}
Here we recall important facts about the $_2F_1(\lambda)_q$ values. The results we require are obtained by
interpreting these values in terms of the trace of Frobenius on the Legendre normal form elliptic curves
$E_{\lambda}^{\Leg}.$ This connection is well known and has been observed by several authors.

\begin{theorem}[Th. 11.10 of \cite{ono-book}]\label{Trace2F1} 
If $\lambda\in\F_q\setminus\{0,1\}$ and $\ch(\F_q)\geq 5,$  then
$$
q\cdot{_2F_1}(\lambda)_q=-\phi(-1)a_\lambda^{\Leg}(q).
$$
\end{theorem}

\begin{remark}  Theorem~\ref{Trace2F1} is analogous to Gauss' classical hypergeometric formula for the real period $\Omega^{\Leg}(\lambda)$ of $E_{\lambda}^{\Leg}$   (for example, see Chapter 9 of  \cite{Husemoller}), where
for $0<\lambda<1$ we have
$$
\pi \cdot{ _2F_1}\left( \begin{matrix} \frac{1}{2} & \frac{1}{2} 
\\ & 1 \end{matrix} \ | \ \lambda\right)= \Omega^{\Leg}(\lambda).
$$
\end{remark}

\subsection{Facts about Legendre normal forms}

As mentioned above, the proof of Theorem~\ref{Theorem2F1} relies on an arithmetic reformulation of the moments of
$_2F_1(\lambda)_q.$  By Theorem~\ref{Trace2F1}, this task requires important facts about the $E_{\lambda}^{\Leg}.$
We now recall these facts.

\begin{proposition}[Proposition 1.7, Chapter III of \cite{silverman}]\label{elliptic-proposition-1}
Let $K$ be a field with $\ch(K)\neq2,3$.

\noindent
(1) Every elliptic curve $E\slash K$ is isomorphic over $\overline{K}$ to an elliptic curve $E^{\Leg}_\lambda.$

\noindent
(2)  If $\lambda\neq0,1$, then the $j$-invariant of $E^{\Leg}_\lambda$ is
$$j(E^{\Leg}_\lambda)=2^8\cdot\frac{(\lambda^2-\lambda+1)^3}{\lambda^2(\lambda-1)^2}.$$

\noindent
(3) The only $\lambda$ for which $j(E^{\Leg}_\lambda)=1728$ are $\lambda=2,-1,$ and $1/2$.

\noindent 
(4) The only $\lambda$ for which $j(E^{\Leg}_\lambda)=0$ are 
$\lambda=\frac{1\pm\sqrt{-3}}{2}$

\noindent
(5)  For every $j\not \in \{0, 1728\},$ the map $K\setminus\{0,1\}\to j(E^{\Leg}_\lambda)$ is six to one. In particular, we have
$$
\left\{\lambda, \frac{1}{\lambda}, 1-\lambda, \frac{1}{1-\lambda}, \frac{\lambda}{\lambda-1},
\frac{\lambda-1}{\lambda}\right\}\to j(E^{\Leg}_\lambda).
$$
\end{proposition}

Since elliptic curves  defined over $\F_q$ with the same $j$-invariant are not necessarily isomorphic over $\F_q,$ we must consider the theory of twists. We only require the standard notion of a quadratic twist.
If $d\in\F_q\setminus\{0,1\},$ and $E$ is given by 
$$
E:\ \ y^2=x^3+a_2x^2+a_4x+a_6,
$$ 
then its quadratic twist $E_d$ is given by\footnote{We note that this choice is equivalent to the usual convention where one has $E_d: \ \ dy^2=x^3+a_2x^2+a_4x+a_6.$}
$$
E_d: \ \ y^2=dx^3+da_2x^2+da_4x+da_6.
$$
If $d$ is a square in $\F_q,$ then $E_d$ is isomorphic to $E$ over $\F_q.$ Moreover, if $p$ is a prime of good reduction for $E_d$ (and hence also $E$), we have that
\begin{equation}\label{TwistTrace}
q+1- |E(\F_q)| =\phi(d)\left(q+1-|E_d(\F_q)|\right).
\end{equation}
The next result characterizes the quadratic twists of Legendre curves with common $j$-invariant.

\begin{proposition}[Prop. 3.2 of \cite{ahlgren-ono}]\label{elliptic-proposition-2}
For $\lambda\in\F_q\setminus\{0,1\},$ the following holds.

\noindent
(1)  $E^\Leg_\lambda$ is the $\lambda$ quadratic twist of $E^\Leg_{1/\lambda}$.

\noindent
(2) $E^\Leg_{\lambda}$ is the $-1$ quadratic twist of $E^\Leg_{1-\lambda}$.

\noindent
(3)  $E^\Leg_{\lambda}$ is the $1-\lambda$ quadratic twist of $E^\Leg_{\lambda/(\lambda-1)}$.

\noindent
(4) $E^\Leg_\lambda$ is the $-\lambda$ quadratic twist of $E^\Leg_{(\lambda-1)/\lambda}.$

\noindent
(5)  $E^\Leg_\lambda$ is the $\lambda-1$ quadratic twist of $E^\Leg_{1/(1-\lambda)}.$
\end{proposition}

By Theorem~\ref{Trace2F1}, we can reformulate the moments of the $_2F_1$  functions as sums over
Legendre normal form elliptic curves. As we shall see in the next subsection, this requires dividing these curves into isomorphism classes over $\F_q$.
To this end, for $\lambda\in\F_q\setminus\{0,1\},$ we define
\begin{equation}\label{Lset}
L(\lambda):=\{\beta\in\F_q\setminus\{0,1\} \ : \ E_\beta^\Leg \cong_{\F_q}E_\lambda^\Leg\}.
\end{equation}
The following three lemmas determine $|L(\lambda)|.$ The first concerns $j\not \in \{0, 1728\}.$

\begin{lemma}\label{counting-LNF-not-supersingular}
If $j(E_\lambda)\not \in \{ 0,1728\},$ then
$$
|L(\lambda)|=\begin{cases}
3 &\ \hbox{if $q\equiv 3\pmod 4$} \\
6 &\ \hbox{if  $q\equiv 1\pmod 4, \lambda$ and $1-\lambda$ are both squares in $\F_q$} \\ 
4 &\ \hbox{if $q\equiv 1\pmod 4,$ either $\lambda$ or $1-\lambda$ is a square in $\F_q$} \\
2 &\ \hbox{if $q\equiv 1\pmod 4,$ neither $\lambda$ nor $1-\lambda$ is a square in $\F_q$}.
\end{cases}
$$
\end{lemma}

\begin{proof} Here we consider the case where
 $q\equiv 3\pmod 4.$ There are exactly two elements of 
 $\{\lambda,1-\lambda,-\lambda,\lambda-1\}$
 that are squares.  
 Therefore, Proposition~\ref{elliptic-proposition-2} applies that $|L(\lambda)|=3.$
The other cases are handled {\it mutatis mutandis.}
\end{proof}

For $j=1728$, we have the following lemma.
\begin{lemma}\label{counting-LNF-supersingular-1} Suppose that  $E_{\lambda}^{\Leg}/\F_q$ has $j(E_{\lambda}^{\Leg})=1728.$

\noindent
(1) If $q\equiv 3\pmod 4,$ then $a_\lambda^\Leg(q)=0.$

\noindent
(2) If  $q\equiv 1\pmod 8,$ then $L(2)=\left\{-1,2,1/2\right\}.$

\noindent
(3)  If $q\equiv 5\pmod 8$, then $L(2)=\left\{-1,2\right\}$ and $L(1/2)=\left\{1/2\right\}.$
\end{lemma}

\begin{proof}
Curves with $j=1728$ have complex multiplication by $\Q(i)$.  There are no ideals in $\Z[i]$ with norm $q\equiv 3\pmod 4$, and so (1) follows easily (for example, see \cite[Section 4]{rosen}). If $q\equiv 1\pmod 4,$ then a similar counting argument as in the proof of  Lemma~\ref{counting-LNF-not-supersingular} gives (2) and (3).
\end{proof}

For $j=0$, we have the following lemma.
\begin{lemma}\label{counting-LNF-supersingular-2} Suppose $E_{\lambda}^{\Leg}/\F_q$  has $j(E_\lambda^\Leg)=0.$

\noindent
(1) There are no such $E_{\lambda}^{\Leg}$ when $q\equiv 2\pmod 3.$

\noindent
(2) If $q\equiv 1\pmod{12},$ then $|L\left(\frac{1\pm\sqrt{-3}}{2}\right)|=2,$ and $\frac{1\pm\sqrt{-3}}{2}$ are squares in $\F_q.$

\noindent
(3) If $q\equiv 7\pmod{12},$ then $|L\left(\frac{1\pm\sqrt{-3}}{2}\right)|=1,$ and $\frac{1\pm\sqrt{-3}}{2}$ are both not squares in $\F_q.$
\end{lemma}

\begin{proof} Claim (1) follows from the unsolvability of
$j(E^{\Leg}_\lambda)=2^8\cdot (\lambda^2-\lambda+1)^3/\lambda^2(\lambda-1)^2=0.$
The proofs of claims (2) and (3) are analogous to the proof of Lemma~\ref{counting-LNF-not-supersingular}.
\end{proof}

To obtain the desired reformulation of the power moments of the $_2F_1$ hypergeometric functions, we make use of the fact that $\Z2\times \Z2 \subseteq E^{\Leg}_{\lambda}(\F_q).$  Our final reformulation makes use of this observation, combined with the fact that certain Hurwitz class numbers enumerate isomorphism classes of elliptic curves with prescribed subgroups and fixed Frobenius traces.

\begin{lemma}\label{ExistenceOfLNF3mod4}
If $q\equiv 3\pmod 4,$ and $E/\F_q$ is an elliptic curve for which $\Z2\times\Z2\subseteq E(\F_q)$, then $E$ is isomorphic to a Legendre normal form elliptic curve over $\F_q.$
\end{lemma}

\begin{proof}
Since $\Z2\times\Z2\subseteq E(\F_q),$ $E$ is given by
$$
E:\ \ \ y^2=(x-\alpha)(x-\beta)(x-\gamma),
$$
where $\alpha, \beta,\gamma\in\F_q.$
After possibly exchanging $\alpha$ and $\beta,$ we may assume that $\beta-\alpha$ is a square.
Under the transformations $y=(\beta-\alpha)^{3/2}Y$ and $ x=(\beta-\alpha)X+\alpha,$ $E\cong E_\lambda^\Leg,$ where $\lambda=\frac{\gamma-\alpha}{\beta-\alpha}.$
\end{proof}

 As the previous lemma indicates, if $q\equiv 3\pmod 4,$ then every
$E/\F_q$ with $\Z2\times\Z2\subseteq E(\F_q)$ is isomorphic over $\F_q$ to a Legendre normal form curve. 
Unfortunately, this is not the case when $q\equiv 1\pmod 4,$ and we call those $E$
without such isomorphic Legendre forms  {\it inconvenient.}

\begin{lemma}\label{ExistenceOfLNF1mod4} Suppose that $q\equiv 1\pmod 4$ and that $E/\F_q$ is inconvenient.

\noindent
(1)  We have that $|E(\F_q)|\not\equiv 0\pmod 8.$

\noindent
(2) There is a $\lambda\in\F_q\setminus\{0,1\}$ and $d\in\F_q,$ where $d\not \in \F_q^2,$ such that $\Z4\times\Z4\subseteq E_\lambda^{\Leg}(\F_q)$  and $E_d\cong E_\lambda^{\Leg}$ over $\F_q.$

\noindent
(3) The phenomenon in (2) induces a bijection between $\F_q$-isomorphism classes of inconvenient curves and those classes for which $\Z4\times\Z4$ is a subgroup of $\F_q$-rational points.
\end{lemma}

\begin{proof}
Let $E$ be an elliptic curve defined by
$$
E:\ \ \ y^2=x(x-\alpha)(x-\beta),
$$
where $\alpha,\beta,\alpha-\beta$ are non-squares in $\F_q.$
The classical 2-descent lemma (for example, see Proposition X.1.4 of \cite{silverman})  
indicates when a rational point $P$ is a double of another rational point, say $Q$.
By our assumptions on $\alpha$ and $\beta$, we find that none of the 2-torsion points are doubles, and so we have that
 $|E(\F_q)|\not\equiv 0\pmod 8.$
Furthermore, the $\alpha$-twist $E_\alpha$ is 
$$
E_\alpha:\ \ \ y^2=\alpha x(x-\alpha)(x-\beta),
$$
and under the transformation $x=\alpha X, y=Y/\alpha^2,$ this is equivalent to
$$
E^\Leg_{\beta/\alpha}:\ \ \ Y^2=X(X-1)(X-\beta/\alpha).
$$
One then applies the 2-descent lemma again.
\end{proof}

We conclude with a classification of those Legendre normal form with $\Z4\times \Z4\subseteq E_{\lambda}^{\Leg}(\F_q).$

\begin{lemma}\label{2TorsionStructure} Suppose that $q\equiv 1\pmod 4$
and $\lambda\in\F_q\setminus\{0,1\}.$ Then we have that $\Z4\times\Z4\subseteq E_\lambda^\Leg(\F_q)$ if and only if $\lambda$ and $1-\lambda$ are both squares in $\F_q.$
\end{lemma}
\begin{proof}
This claim follows easily again by the 2-descent lemma. \end{proof}

\subsection{Isomorphism classes of elliptic curves with prescribed subgroups}

We have reformulated the  moments of the $_2F_1$  functions as sums over isomorphism classes of elliptic curves for which $\Z2\times\Z2 \subseteq E(\F_q)$. Therefore, we  seek formulas for the number of such classes. Thankfully, these are known
due to work of Schoof \cite{schoof}, and they involve Hurwitz class numbers.

To make this precise, we first recall some notation.
 If $-D<0$ such that $-D\equiv0,1\pmod{4},$ then denote by $\mathcal{O}(-D)$ the unique imaginary quadratic order with discriminant $-D.$ Let $h(D)=h(\mathcal{O}(-D))$ denote\footnote{We note that $H(D)=H^*(D)=h(D)=0$ whenever $-D$ is neither zero nor a negative discriminant.} the order of the class group of $\mathcal{O}(-D)$ and let $\omega(D)=\omega(\mathcal{O}(-D))$ denote half the number of roots of unity in $\mathcal{O}(-D).$ Furthermore, define
\begin{equation}
H(D):=\sum\limits_{\mathcal{O}\subseteq\mathcal{O'}\subseteq\mathcal{O}_{\text{max}}}h(\mathcal{O'})
\ \ \ {\text {\rm and}}\ \ \ 
\ \ H^{\ast}(D):=\sum\limits_{\mathcal{O}\subseteq\mathcal{O'}\subseteq\mathcal{O}_{\text{max}}}\frac{h(\mathcal{O'})}{\omega(\mathcal{O'})},
\end{equation}
where the sum is over all orders $\mathcal{O'}$ between $\mathcal{O}$ and the maximal order $\mathcal{O}_{\text{max}}.$
The following theorem of Schoof \cite{schoof} gives the results we require.

\begin{theorem}[Section 4 of \cite{schoof}]\label{Schoof} If $p\geq 5$ is prime, and  $q=p^r,$ then the following are true.

\noindent
(1) If $n\geq 2$ and  $s$ is a nonzero integer for which  $p|s$ and $s^2\neq 4q,$  then there are no elliptic curves $E/\F_q$
with $|E(\F_q)|=q+1-s$ and 
  $\Z n\times \Z n\subseteq E(\F_q).$

\noindent
(2)  If $r$ is even and $s=\pm 2p^{r/2},$ then the number of isomorphism classes of elliptic curves over $\F_q$ with
$\Z2 \times \Z2 \subseteq E(\F_q)$ and
 $|E(\F_q)|=q+1-s$ is
\begin{equation}\label{Ap}
S(p):=\frac{1}{12}\left(p+6-4\leg{-3}{p}-3\leg{-4}{p}\right),
\end{equation}
where $\leg{\cdot}{p}$ is the Legendre symbol.

\noindent
(3)  Suppose that $n$ and $s$ are integers such that $s^2\leq 4q,$ $p\nmid s,$ $n^2\mid (q+1-s),$ and $n\mid (q-1).$ Then the number of isomorphism classes of elliptic curves over $\F_q$ with
$|E(\F_q)|=q+1-s$ and $\Z n\times\Z n\subseteq E(\F_q)$ is $H\left(\frac{4q-s^2}{n^2}\right).$
\end{theorem}

\begin{remark} Theorem~\ref{Schoof} is a compilation of various results from \cite{schoof}.
Namely, (1) follows from Theorem 4.2 (ii-iii) and Lemma 4.8 (i). Claim
(2) follows from Theorem 4.6 and Lemma 4.8 (ii).
Finally, (3) is a consequence of the proof of Theorem 4.9 (i).
\end{remark}

\begin{remark} The number $S(p)$ defined in(\ref{Ap}) 
also happens to be
the number of isomorphism classes of supersingular elliptic curves over $\overline{\F}_p$ (for example, see Proposition 2.49 of \cite{ono-book}).
\end{remark}

\subsection{Formulas for $_2F_1$ moments}

Finally, we assemble the results of the previous subsections to obtain the desired weighted class number sum expressions for the power moments.

\begin{proposition}\label{Moments2F1} Suppose that $p\geq 5$ is prime.
 If $r$ and $m$ are positive integers, then the following are true for $q=p^r,$ where in each summation we have that
 $-2\sqrt{q}\leq s\leq 2\sqrt{q}.$

\noindent
(1)  If $r$ is odd and $m$ is even, then we have
$$
q^m\sum\limits_{\lambda\in\F_q}{_2F_1}(\lambda)_q^m=1+3\sum\limits_{\substack{\gcd(s,p)=1 \\ s\equiv q+1\pmod 4}}
H^\ast\left(\frac{4q-s^2}{4}\right)s^{m}.
$$

\noindent
(2) If $r$ and $m$ are both even, then there is a rational number $C(q)\in [0,6]$ for which
$$
q^m\sum\limits_{\lambda\in\F_q}{_2F_1}(\lambda)_q^m=1+C(q) S(p)\cdot q^{m/2}+3\sum\limits_{\substack{\gcd(s,p)=1 \\ s\equiv q+1\pmod 4}}H^\ast\left(\frac{4q-s^2}{4}\right)s^{m}.
$$

\noindent
(3)  If $q\equiv 3\pmod 4$ and $m$ is odd, then we have
$
q^m\sum\limits_{\lambda\in\F_q}{_2F_1}(\lambda)_q^m=1.
$

\noindent
(4)  If $q\equiv 1\pmod 4$ and $m$ is odd, then there is a rational number $D(q)\in [-6,6]$ for which
\begin{displaymath}
\begin{split}
&q^m\sum\limits_{\lambda\in\F_q}{_2F_1}(\lambda)_q^m\\
&\ \ =-1-2\sum\limits_{\substack{\gcd(s,p)=1 \\ s\equiv q+1\pmod {8}}}H^\ast\left(\frac{4q-s^2}{4}\right)s^m-4\sum\limits_{\substack{\gcd(s,p)=1 \\ s\equiv q+1\pmod {16}}} H^\ast\left(\frac{4q-s^2}{16}\right)s^m-D(q)S(p) q^{m/2}.
\end{split}
\end{displaymath}
\end{proposition} 

\begin{remark} 
The rational number $C(q)$ is the average number of Legendre form  curves in an $\F_q$-isomorphism class with $a^{\Leg}(q)_\lambda=\pm 2\cdot p^{r/2}.$ Similarly, $D(q)$ is the average number of such curves in an isomorphism class with $a^{\Leg}_\lambda(q)=2p^{r/2}$ minus the average number  with $a^{\Leg}_\lambda(q)=-2 p^{r/2}.$
\end{remark}

\begin{proof}
We first prove (3) as it is a triviality.
By Theorem 4.4 of \cite{greene}, if $q\equiv 3\pmod 4$ and $\lambda\in\F_q\setminus\{0,1\},$ then ${_2F_1}(\lambda)_q=-{_2F_1}(1-\lambda)_q.$
Therefore, claim (3) follows from the resulting cancellation, combined with the fact that
$_2F_1(1)_q=1/q$ and $_2F_1(0)_q=0.$

The proofs of claims (1), (2), and (4)  are very similar. Therefore, we only prove (4) for brevity. We make use of Theorem~\ref{Schoof}, and Lemmas~\ref{counting-LNF-not-supersingular} through \ref{2TorsionStructure}.
Using Theorem~\ref{Trace2F1}, we rewrite the sum in terms of $-a_\lambda^\Leg(q)$. We then decompose the sum
$$
-\sum\limits_{\lambda\in\F_q\setminus\{0,1\}}a_\lambda(q)^m=-\sum\limits_s|I(s,q)|\cdot s^m,
$$
where $I(s,q)=\left\{\lambda\in\F_q\setminus\{0,1\} \ : \ a_\lambda^\Leg(q)=s\right\}.$
By Theorem~\ref{Schoof} and Lemmas~\ref{counting-LNF-not-supersingular}-\ref{counting-LNF-supersingular-2}, and Lemma~\ref{2TorsionStructure}, we have

\begin{align}
-\sum\limits_{\lambda\in\F_q\setminus\{0,1\}} a_q(\lambda)^m&=-4\sum\limits_{\substack{\gcd(s,p)=1 \\ s\equiv q+1\pmod {8}}}
\left[H^\ast\left(\frac{4q-s^2}{4}\right)-H^\ast\left(\frac{4q-s^2}{16}\right)\right]s^m\notag\\
&-2\sum\limits_{\substack{\gcd(s,p)=1 \\ s\not\equiv q+1\pmod {8}}}H^\ast\left(\frac{4q-s^2}{4}\right)s^m-6\sum\limits_{\substack{\gcd(s,p)=1 \\ s\equiv q+1\pmod {16}}}H^\ast\left(\frac{4q-s^2}{16}\right)s^m\notag\\
&-|I(2q^{1/2},q)|\cdot 2q^{m/2}-|I(-2q^{1/2},q)|\cdot (-2q^{m/2})+E(q,m),\notag
\end{align} 
where $E(q,m)$ is the sum over equivalence classes which do not contain a Legendre normal form.
However, by Lemma~\ref{ExistenceOfLNF1mod4}, we see that
$$
E(q,m)=2\sum\limits_{\substack{\gcd(s,p)=1 \\ -s\equiv q+1\pmod{16}}}H^\ast\left(\frac{4q-s^2}{16}\right)s^m=2\sum\limits_{\substack{\gcd(s,p)=1 \\ s\equiv q+1\pmod{16}}}H^\ast\left(\frac{4q-s^2}{16}\right)(-s)^m.
$$
The result follows by considering congruence conditions and the fact that $m$ is odd.
\end{proof}

\section{The $_3F_2(\lambda)_q$ and the arithmetic of $E^{\Cl}_{\lambda}$}\label{Section3F2}

Here we recall important facts about the $_3F_2(\lambda)_q$ values, which are related to
the squares of the trace of Frobenius for the Clausen elliptic curves $E_{\lambda}^{\Cl}.$

\begin{theorem}[Th. 5 of \cite{ono}]\label{Trace3F2}
If $\lambda\in\F_q\setminus\{0,-1\},$ $\ch(\F_q)\geq 5$ and $a^{\Cl}_\lambda(q):=q+1-|E_\lambda^{\Cl}(\F_q)|,$ then we have
$$
q+q^2\phi(\lambda+1)\cdot{_3F_2}\left(\frac{\lambda}{\lambda+1}\right)_q=a^{\Cl}_\lambda(q)^2.
$$
\end{theorem}

\begin{remark}  Theorem~\ref{Trace3F2} has a counterpart in terms of classical hypergeometric functions. For $0<\lambda<1,$ if $\Omega^{\Cl}(\lambda)$ is the real period of $E_{\lambda}^{\Cl},$ then
McCarthy \cite{McCarthy} proved that
$$
_3F_2\left(\begin{matrix}\onehalf &\onehalf &\onehalf\\ \ & 1\ & 1
\end{matrix} \ | \ \frac{\lambda}{\lambda+1}\right)=\frac{\sqrt{1+\lambda}}{\pi^2}\cdot \Omega^{\Cl}(\lambda)^2.
$$
\end{remark}

\subsection{Certain moments of traces of Frobenius of the Clausen elliptic curves}

The goal of this subsection is to obtain two types of power moments for the Clausen curves.
To this end, we first fix some notation. 
We let $\mathcal{C}$ denote a generic isomorphism class of elliptic curves over $\F_q,$ 
where throughout  $p\geq 5$ is prime and $q=p^r$, where $r$ is a fixed positive integer. We let $\mathcal{I}_{q}$ denote the set of all isomorphism classes of elliptic curves over $\F_{q},$ and define
\begin{equation}
I(s,q):=\left\{\mathcal{C}\in\mathcal{I}_{q} \ : \ \ \forall \ \ E\in\mathcal{C} \ {\text {\rm we have}}\ |E(\F_{q})|=q+1\pm s\right\},
\end{equation}
\begin{equation}
I_2(s,q):=\left\{\mathcal{C}\in I(s,q) \ : \ \forall \ E\in\mathcal{C}\ {\text {\rm we have}}\  E(\F_{q})[2]\cong\Z2\times\Z2\right\},
\end{equation}
where $0<s\leq 2\sqrt{q}$ is even. 
We recall that the size of $I(s,q)$ is given by Theorem~\ref{Schoof} as
$$
|I(s,q)|=\begin{cases}
2H(4q-s^2)  &\text{if } p\nmid s \\
2\cdot S(p) & \text{if } s^2=4q \text{ and } r\text{ is even } \\
0 &\text{otherwise,}
\end{cases}
$$
where  $S(p)$ is given by (\ref{Ap}).

For even $0<s\leq 2\sqrt{q},$ we  let
\begin{equation}
L(s,q)=\left\{\lambda\in\F_{q}\setminus\{0,-1\} \ : \ ~ a^\Cl_\lambda(q)=\pm s\right\}.
\end{equation}
The following proposition about most isomorphism classes with nonzero even traces of Frobenius will simplify our later calculations.

\begin{proposition}\label{EvenMomentsClausen} If $0<s\leq 2\sqrt{q}$ is even, $1/3,-1/9\not\in L(s,q),$ and $|E(\F_q)|\not \in \{q+1\pm s\}$ for any elliptic curve $E/\F_q$ with $j(E)=1728,$ then the following is true.

\noindent
(1) If $n$ is a positive integer, then
$$
\sum\limits_{\substack{\lambda\in\F_{q}\setminus\{0,-1\} \\ a^\Cl_\lambda(q)=\pm s}} a^\Cl_\lambda(q)^{2n}=s^{2n}\cdot\left(\frac{1}{2}\cdot|I(s,q)|+|I_2(s,q)|\right).
$$

\noindent
(2) If $n$ is a positive integer, then
$$
\sum\limits_{\substack{\lambda\in\F_{p^r}\setminus\{0,-1\} \\ a^\Cl_\lambda(q)=\pm s}} \phi(-\lambda)a^\Cl_\lambda(q)^{2n}=s^{2n}\cdot\left(-\frac{1}{2}\cdot|I(s,q)|+2\cdot|I_2(s,q)|\right).
$$
\end{proposition}
\begin{proof}
As  $I(s,q)$ includes quadratic twists, we let $\mathcal{C}^{\text{tw}}$ be the isomorphism class of quadratic twists of curves in $\mathcal{C}$ by nonsquares in $\F_{q},$ which then gives 
$I(s,q)=\left\{\mathcal{C}_1,\ldots,\mathcal{C}_h,\mathcal{C}_1^{\text{tw}},\ldots,\mathcal{C}_h^{\text{tw}}\right\}.$
To study these isomorphism clases, it is convenient to then define
$$
\widetilde{I}(s,q)=\left\{\mathcal{C}_1\cup\mathcal{C}_1^{\text{tw}},\ldots,\mathcal{C}_h\cup\mathcal{C}_h^{\text{tw}}\right\}\ \ \ {\text {\rm and}}\ \ \ 
\widetilde{I}_2(s,q)=\left\{\mathcal{C}\cup\mathcal{C}^{\text{tw}} \ : \ ~\mathcal{C}\in I(s,q)\right\}.
$$

 By Theorem~\ref{Schoof}, we can compute $|\widetilde{I}(s,q)|$ and $|\widetilde{I_2}(s,q)|.$ Therefore, we aim to relate the cardinalities of $L(s,q), \widetilde{I}(s,q)$ and $\widetilde{I}_2(s,q).$ To this end, we define 
 $F:L(s,q)\to\widetilde{I}(s,q)$
  by 
$F(\lambda):=[E^\Cl_\lambda]\cup[E^\Cl_{\lambda}]^{\text{tw}},$
where $[E]$ is the $\F_q$-isomorphism class of elliptic curves containing $E$.

By Lemma 7.1 \cite{fop},
$F$ is surjective, unless $j(E)=1728,$ in which case $F$ misses exactly one isomorphism class. Furthermore, by Lemma 7.2 of \cite{fop}, if $1/3,-1/9\not\in L(s,q),$ then $F$ is three-to-one if and only if $-\lambda$ is a square in $\F_q,$ and is one-to-one otherwise.
To see (1), we note that the above discussion gives that
$$
\sum\limits_{\substack{\lambda\in\F_{q}\setminus\{0,-1\} \\ a^\Cl_\lambda(q)=\pm s}} a^\Cl_\lambda(q)^{2n}=s^{2n}\cdot\left(\frac{1}{2}\left(|I(s,q)|-|I_2(s,q)|\right)+\frac{3}{2}\cdot|I_2(s,q)|\right).
$$
Similarly, to obtain (2), the above discussion gives that
$$
\sum\limits_{\substack{\lambda\in\F_{q}\setminus\{0,-1\} \\ a^\Cl_\lambda(q)=\pm s}} \phi(-\lambda)a^\Cl_\lambda(q)^{2n}=s^{2n}\cdot\left(-\frac{1}{2}\left(|I(s,q)|-|I_2(s,q)|\right)+\frac{3}{2}\cdot|I_2(s,q)|\right).
$$
These two claims clearly reduce to (1) and (2) respectively.
\end{proof}

The discussion above also provides the following critical bound for $|L(s,q)|.$

\begin{proposition}\label{ClausenError}
If $0<s\leq 2\sqrt{q}$ is even, then we have
$|L(s,q)|\leq 3\cdot\max\left\{H(4q-s^2), S(p),2\right\}.$
\end{proposition}

\section{Harmonic Maass forms and weighted sums of Fourier coefficients}\label{HarmonicMaassForms}

In this section we explain how the weighted sums of class numbers in the previous section arise naturally in the theory of harmonic Maass forms (for background, see \cite{BFOR}).
The connection with harmonic Maass forms stems from the following well-known theorem about Zagier's weight 3/2 nonholomorphic Eisenstein series.

\begin{theorem}[\cite{Zagier1}]\label{ZagierSeries}
The function
$$
\mathcal{H}(\tau)=-\frac{1}{12}+\sum\limits_{n=1}^\infty H^\ast(n)q_\tau^n+\frac{1}{8\pi\sqrt{y}}+\frac{1}{4\sqrt{\pi}}\sum\limits_{n=1}^\infty n\Gamma(-\frac{1}{2}; 4\pi n^2y)q^{-n^2},
$$
where $\tau=x+iy\in \HH$ and $q_\tau:=e^{2\pi i\tau},$ is a weight $3/2$ harmonic Maass form with manageable growth at the cusps of $\Gamma_0(4).$ 
\end{theorem}

This theorem asserts that the generating function for Hurwitz class numbers \footnote{Here we adopt the convention that $H^*(0):=-1/12.$}  is the {\it holomorphic part} of the harmonic Maass form $\mathcal{H}(\tau).$  More generally
(for example, see Lemma 4.3 of \cite{BFOR}), every weight $k\neq 1$ harmonic weak Maass form $f(\tau)$ has a Fourier expansion of the form
\begin{equation}\label{HMFFourier}
f(\tau)=f^{+}(\tau)+\frac{(4\pi y)^{1-k}}{k-1}\overline{c_f^{-}(0)}+f^{-}(\tau),
\end{equation}
where 
\begin{equation}\label{HMSParts}
f^{+}(\tau)=\sum\limits_{n=m_0}^\infty c_f^{+}(n)q_\tau^n \ \ \
{\text {\rm and}}\ \ \ 
f^{-}(\tau)=\sum\limits_{\substack{n=n_0\\ n\neq 0}}^\infty \overline{c_f^{-}(n)}n^{k-1}\Gamma(1-k;4\pi |n| y)q_\tau^{-n}.
\end{equation}
Here $\Gamma(\alpha;x):=\int_{\alpha}^{\infty}e^{-t}t^{x-1}dt$ is the usual incomplete Gamma-function.
The function $f^{+}(\tau)$ is called the {\it holomorphic part} of $f.$

The weighted sums of class numbers we require appear in formulas for the Fourier coefficients of certain families of nonholomorphic modular forms. These forms are constructed from Zagier's $\mathcal{H}(\tau)$ as a simple implementation of the Rankin-Cohen bracket operators, which are combinatorial expressions in derivatives of pairs of modular forms. This method was previously applied by Mertens \cite{mertens, mertensRMS} in his proof of a deep conjecture of Cohen on the Cohen-Eisenstein series.

\subsection{Combinatorial Interlude}

To carry out the strategy described above, we require a framework of combinatorial identities
for the degree $a-2$
 homogeneous polynomials
\begin{equation}
P_{a,b}(X,Y):=\sum\limits_{j=0}^{a-2}\binom{j+b-2}{j}X^j(X+Y)^{a-j-2},
\end{equation}
where $a\geq 2$ is a positive integer and $b$ is any real number. This framework captures
the nonholomorphic modular forms constructed with the Rankin-Cohen brackets.
The next proposition gives a significant identity for certain complicated algebraic expressions in these
polynomials.

\begin{proposition}\label{ExplicitPiHol-Poly}
If $m>n$ are positive integers, then we have
\begin{displaymath}
\begin{split}
&2^{-2\nu-1}\binom{2\nu+1}{\nu+1}\left(m^{-\frac{1}{2}}(m^{\frac{1}{2}}-n^{\frac{1}{2}})^{2\nu+2}\right)=\\
&\ \ \ \ \ \ \ \ \ \ \ \ \ \sum\limits_{\mu=0}^\nu\binom{\frac{1}{2}+\nu}{\nu-\mu}\binom{\frac{1}{2}+\nu}{\mu}m^{\nu-\mu}\times\left(m^{\mu-2\nu-1/2}P_{3+2\nu,\frac{1}{2}-\mu}(m-n,n)-n^{\frac{1}{2}+\mu}\right).
\end{split}
\end{displaymath}
\end{proposition}

\noindent
This proposition is analogous to Proposition V.2.7 of \cite{mertens}. Moreover, its proof follows along the same lines.
The key lemmas we require are as follows, where binomial  and multinomial coefficients with non-integral arguments are defined using the Gamma-function.

\begin{lemma}\label{CombLemma}
If 
$\nu\geq 1,$ and $j\geq 0$ are integers, then we have
$$
\sum\limits_{\mu=0}^\nu\frac{(-1)^\mu}{\mu-j+\frac{1}{2}}\cdot \binom{4\nu-2\mu+1}{2\nu-\mu,\ \mu,\ 2\nu-2\mu+1}=2^{4\nu+2}(-1)^j\frac{(2\nu-j+1)!j!}{(2j)!(2\nu-2j+2)!}.
$$
\end{lemma} 
\begin{proof}[Sketch of the Proof] This claim is analogous to Lemma~V.2.6 of \cite{mertens}, which stems from an expression of the form
$$
\sum\limits c_\mu=\frac{1}{-j+\frac{1}{2}}\cdot \binom{4\nu-1}{2\nu}\cdot {_3F_2}\left( \begin{matrix} -\nu & -j+\frac{1}{2} & -\nu+\frac{1}{2}\\
\ & -j+\frac{3}{2} & -2\nu+\frac{1}{2}\end{matrix} \ | \ 1\right).
$$
The claim is obtained by applying the same steps to the following expression with sign changes
$$
\sum\limits c_\mu=\frac{1}{-j+\frac{1}{2}}\cdot \binom{4\nu+1}{2\nu}\cdot {_3F_2}\left( \begin{matrix} -\nu & -j+\frac{1}{2} & -\nu-\frac{1}{2}\\
\ & -j+\frac{3}{2} & -2\nu-\frac{1}{2}\end{matrix} \ | \ 1\right).
$$
\end{proof}

\begin{lemma}\label{CombLemma2Equation}
The following are true:

\noindent
(1) If $\mu\leq\nu$ are nonnegative integers, then we have
$$
\binom{\nu+\frac{1}{2}}{\nu-\mu}\binom{\nu+\frac{1}{2}}{\mu}=2^{-2\nu-1}\binom{2\nu+1}{\nu+1}\binom{2\nu+2}{2\mu+1}.
$$

\noindent
(2)  If $0\leq \mu\leq \nu$ and $j\geq 0$ are integers, then we have
$$
\binom{2\nu-\mu+\frac{1}{2}}{2\nu+1-j}\binom{j-\mu-\frac{3}{2}}{j}=\frac{(-1)^{\mu+1}}{j-\mu-\frac{1}{2}}\cdot 2^{-4\nu-2}\frac{(4\nu-2\mu+1)!(2\mu+1)!}{(2\nu-\mu)! \ \mu ! \ j! \ (2\nu-j+1)!}.
$$
\end{lemma}

\begin{proof}[Sketch of Proof]
To prove (1), we  emulate Mertens' proof (see p. 60 of \cite{mertens}) that
$$
\binom{\nu+\frac{1}{2}}{\nu-\mu}\binom{\nu-\frac{1}{2}}{\mu}=2^{-2\nu}\binom{2\nu}{\nu}\binom{2\nu+1}{2\mu+1}.
$$
He gives explicit steps involving standard properties of the Gamma-function that transform the left-hand side into the right-hand side.
To obtain (1), one applies the same steps to
$$
\binom{\nu+\frac{1}{2}}{\nu-\mu}\binom{\nu+\frac{1}{2}}{\mu}.
$$

To prove (2), we emulate Mertens' proof (see p. 61 of \cite{mertens}) that
$$
\binom{2\nu-\mu-\frac{1}{2}}{2\nu-j}\binom{j-\mu-\frac{3}{2}}{j}=\frac{(-1)^{\mu+1}}{j-\mu-\frac{1}{2}}\cdot 2^{-4\nu}\frac{(4\nu-2\mu-1)!(2\mu+1)!}{(2\nu-\mu-1)! \ \mu ! \ j! \ (2\nu-j)!}.
$$
He gives explicit steps which transform the left-hand side into the right-hand side.
To obtain (2), one applies the same steps to
$$
\binom{2\nu-\mu+\frac{1}{2}}{2\nu+1-j}\binom{j-\mu-\frac{3}{2}}{j}.
$$
\end{proof}

\noindent
Finally, we recall an important identity for the  polynomials $P_{a,b}(X,Y)$ obtained by Mertens.

\begin{lemma}[Lemma V.1.8 of \cite{mertens}]\label{PabFormula}
If $b\neq 1,2,$ then
\begin{equation}\label{IdentityP}
P_{a,b}(X,Y)=\sum\limits_{j=0}^{a-2}\binom{a+b-3}{a-2-j}\binom{j+b-2}{j}(X+Y)^{a-2-j}(-Y)^j.
\end{equation}
\end{lemma}

Using these lemma above, we are now able to prove Proposition~\ref{ExplicitPiHol-Poly}.

\begin{proof}[Proof of  Proposition~\ref{ExplicitPiHol-Poly}]
To prove the proposition, we begin with the right-hand side of the claimed formula.
To start, we absorb the powers of $m$ by
\begin{displaymath}
\begin{split}
&\sum\limits_{\mu=0}^\nu\binom{\frac{1}{2}+\nu}{\nu-\mu}\binom{\frac{1}{2}+\nu}{\mu}m^{\nu-\mu}\times\left(m^{\mu-2\nu-\frac{1}{2}}P_{3+2\nu,\frac{1}{2}-\mu}(r,n)-n^{\frac{1}{2}+\mu}\right)\\ 
&\ \ \ \ \ \ \ \ \ \ \ \ =\sum\limits_{\mu=0}^\nu\binom{\nu +\frac{1}{2}}{\nu-\mu}\binom{\nu+\frac{1}{2}}{\mu}\left(m^{-\nu-\frac{1}{2}}P_{3+2\nu,\frac{1}{2}-\mu}(r,n)-n^{\frac{1}{2}+\mu}m^{\nu-\mu}\right).
\end{split}
\end{displaymath}
By combining Lemma~\ref{CombLemma2Equation} (1) with Lemma~\ref{PabFormula},  one obtains
$$
=2^{-2\nu-1}\binom{2\nu+2}{\nu+1}\left(\sum\limits_{\mu=0}^{\nu}\sum\limits_{j=0}^{2\nu+1}m^{\nu-j+\frac{1}{2}}\binom{2\nu-\mu+\frac{1}{2}}{2\nu+1-j}\binom{j-\mu-\frac{3}{2}}{j}(-n)^j-\sum\limits_{\mu=0}^\nu\binom{2\nu+2}{2\mu+1}n^{\frac{1}{2}+\mu}m^{\nu-\mu}\right).
$$
By the Binomial Theorem, we note that the left-hand side of the claim is
$$
m^{-\frac{1}{2}}(m^{\frac{1}{2}}-n^{\frac{1}{2}})^{2\nu+2}=\sum\limits_{\mu=0}^{\nu+1}\binom{2\nu+2}{2\mu}m^{\nu-\mu+\frac{1}{2}}n^{\mu}-\sum\limits_{\mu=0}^\nu\binom{2\nu+2}{2\mu+1}n^{\mu+\frac{1}{2}}m^{\nu-\mu}.
$$
Therefore, it suffices to show that
$$
\sum\limits_{\mu=0}^{\nu}\sum\limits_{j=0}^{2\nu+1}m^{\nu-j+\frac{1}{2}}\binom{2\nu-\mu+\frac{1}{2}}{2\nu+1-j}\binom{j-\mu-\frac{3}{2}}{j}(-n)^j=\sum\limits_{\mu=0}^\nu\binom{2\nu+2}{2\mu}m^{\nu-\mu+\frac{1}{2}}n^\mu.
$$
Lemma~\ref{CombLemma2Equation} (2), followed by an application Lemma~\ref{CombLemma}, implies this equality.
\end{proof}

\subsection{Families of modular forms obtained from Rankin-Cohen brackets}

As alluded to earlier, the weighted sums of class numbers we require arise in formulas for the coefficients of
certain  families of non-holomorphic modular forms. These families are obtained from Zagier's $\mathcal{H}(\tau)$ by making use of Rankin-Cohen brackets. In this section we recall several important facts about the nonholomorphic modular forms
obtained by this method, along with their holomorphic modular form images under the process of holomorphic projection.

To make this precise, let
 $f$ and $g$ be smooth functions defined on the upper-half  of the complex plane $\HH$, and let $k, l\in\R_{>0}$ and $\nu\in\N_0.$ The $\nu$th {\it Rankin-Cohen bracket} of $f$ and $g$ is 
\begin{equation}\label{RankinCohenBracket}
[f,g]_\nu:=\frac{1}{(2\pi i)^\nu}\sum\limits_{r+s=\nu}(-1)^r\binom{k+\nu-1}{s}\binom{l+\nu-1}{r}\frac{d^r}{d\tau^r}f\cdot\frac{d^s}{d\tau^s}g.
\end{equation}
As the next proposition illustrates, these operators preserve modularity.

\begin{proposition}[Th. 7.1 of \cite{cohen}]\label{BracketProposition}
Let $f$ and $g$ be (not necessarily holomorphic) modular forms of weights $k$ and $l,$ respectively on a congruence subgroup $\Gamma.$ Then the following are true.

\noindent
(1) We have that $[f,g]_\nu$ is modular of weight $k+l+2\nu$ on $\Gamma.$ 

\noindent
(2) If $\gamma\in SL_2(\R),$ then
under the usual modular slash operator we have
$$
[f|_k\gamma,g|_l\gamma]_\nu=([f,g]_\nu)|_{k+l+2\nu}\gamma.
$$
\end{proposition}

\begin{remark}
Proposition~\ref{BracketProposition} (2)  is important for studying the behavior of Rankin-Cohen brackets at cusps. It shows that if $f$ and $g$ are smooth functions that do not blow up at any cusp, and $[f,g]_\nu$ vanishes at the cusp $i\infty,$ then it vanishes at all other cusps for $\nu>0.$
\end{remark}

By Proposition~\ref{BracketProposition}, we have a procedure for producing many nonholomorphic modular forms from derivatives of a pair of seed forms $f$ and $g$. We shall study forms that arise in this way from $f(\tau):=\mathcal{H}(\tau)$ and certain univariate theta functions for $g(\tau).$ To prove our results, we make use of canonical holomorphic modular forms that have coefficients with the same asymptotic properties as $[f,g]_{\nu}.$ These forms are obtained by the method of holomorphic projection.

To make this precise,
suppose $f:\HH\rightarrow\C$ is a (not necessarily holomorphic) modular form of weight $k\geq 2$ on a congruence subgroup $\Gamma$ with Fourier expansion
$$
f(\tau)=\sum_{n\in\Z}c_f(n,y)q_{\tau}^n,
$$
where $\tau=x+iy.$ Let $\{\kappa_1,\ldots, \kappa_M\}$ be the cusps of $\Gamma,$ where  $\kappa_1:=i\infty.$ Moreover, for each $j$ let $\gamma_j\in \SL_2(\Z)$ satisfy
$\gamma_j\kappa_j=i\infty.$ Then suppose the following are true.

\noindent
(1) There is an $\varepsilon>0$ and a constant $c_0^{(j)}\in\C$ for which
$$
f\left(\gamma_j^{-1}w\right)\left(\frac{d\tau}{dw}\right)^{k/2}=c_0^{(j)}+O(\im(w))^{-\varepsilon},
$$
for all $j=1,\ldots,M$ and $w=\gamma_j\tau.$

\noindent
(2) For all $n>0,$ we have that $c_f(n,y)=O(y^{2-k})$ as $y\rightarrow0.$
Then, the {\it holomorphic projection of $f$} is defined by
\begin{equation}
(\pi_{\text{hol}}f)(\tau):=c_0+\sum\limits_{n=1}^{\infty}c(n)q_{\tau}^n,
\end{equation}
where $c_0=c_0^{(1)}$ and for $n\geq1$
$$
c(n)=\frac{(4\pi n)^{k-1}}{(k-2)!}\int_{0}^{\infty}c_f(n,y)e^{-4\pi ny}y^{k-2}dy.
$$
The following proposition explains the important role of the projection operator.
\begin{proposition}[Prop. 10.2 of \cite{BFOR}]\label{HolProjProp}
Assuming the hypotheses above, if $k>2$ (resp. $k=2$), then $\pi_{\text{hol}}(f)$ is a weight $k$ holomorphic modular form (resp. weight $2$ quasimodular form) on $\Gamma.$
\end{proposition}

Turning to the setting we consider, suppose that $f$ is a harmonic Maass form of weight $k\in \frac{1}{2}\Z$ on $\Gamma_0(N)$ with manageable growth at the cusps, and that $g$ is a holomorphic modular form of weight $l$ on $\Gamma_0(N).$ 
Moreover, suppose that $[f,g]_{\nu}$ satisfies the hypothesis in the definition of holomorphic projection. By additivity,  the holomorphic modular form obtained by Proposition~\ref{HolProjProp} has the following convenient decomposition
\begin{equation}\label{HolomorphicProjectionDecomposition}
\pi_{\text{hol}}([f,g]_\nu)=[f^{+},g]_\nu+\frac{(4\pi)^{1-k}}{k-1}\overline{c_f^{-}(0)}\pi_{\text{hol}}([y^{1-k},g]_\nu)+\pi_{\text{hol}}([f^{-},g]_\nu).
\end{equation}

For our applications, the weighted class number sums will arise from the first summand $[f^{+},g]_\nu$ of 
(\ref{HolomorphicProjectionDecomposition}), when $g(\tau)$ is a univariate theta function, and
$f(\tau)=\mathcal{H}(\tau).$
This term $[\mathcal{H},g]_\nu$ clearly involves weighted sums of class numbers via
Theorem~\ref{ZagierSeries}.

The other two summands in (\ref{HolomorphicProjectionDecomposition}) must be bounded for our applications.
The next lemma offers a closed formula for the Fourier expansion of the middle term.

 \begin{lemma}[Lemma V.1.4 of \cite{mertens}]\label{HolProjExplicit1}
Assuming the hypotheses above, if $g(\tau)$ has Fourier series
$g(\tau)=\sum_{n=0}^{\infty}a_g(n)q_{\tau}^n,$
 then we have
$$
\frac{(4\pi)^{1-k}}{k-1}\pi_{\text{hol}}([y^{1-k},g]_\nu)=\kappa(k,l,\nu)\cdot \sum\limits_{n=0}^\infty n^{k+\nu-1}a_g(n)q_\tau^n,
$$
where
$$
\kappa(k,l,\nu):=\frac{1}{(k+l+2\nu-2)!(k-1)}\sum\limits_{\mu=0}^\nu\left(\frac{\Gamma(2-k)\Gamma(l+2\nu-\mu)}{\Gamma(2-k-\mu)}\binom{k+\nu-1}{\nu-\mu}\binom{l+\nu-1}{\mu}\right).
$$
\end{lemma}

Finally, the last term in (\ref{HolomorphicProjectionDecomposition})
 can be bounded thanks to the following theorem of Mertens that offers a closed
formula 
in terms of the Fourier coefficients of $f$ and $g.$
\begin{theorem}[Th. V.1.5 of \cite{mertens}]\label{HolProjExplicit2}
If  $c_f^{-}(n)$ and $a_g(n)$ are bounded polynomially, then we have
$\pi_{\text{hol}}([f^{-},g]_\nu=\sum\limits_{r=1}^\infty b(r)q_\tau^{r},$
where 

\begin{displaymath}
\begin{split}
b(r)=-\Gamma(1-k)&\sum\limits_{m-n=r}a_g(m)\overline{c^{-}_f(n)}\sum\limits_{\mu=0}^\nu\binom{k+\nu-1}{\nu-\mu}\binom{l+\nu-1}{\mu}m^{\nu-\mu}\\ 
&\ \ \ \ \ \ \ \ \ \ \ \ \ \ \ \ \ \ \ \ \times \left(m^{\mu-2\nu-l+1}P_{k+l+2\nu,2-k-\mu}(r,n)-n^{k+\mu-1}\right),
\end{split}
\end{displaymath}

where the sum runs over positive integers $m$ and $n.$
\end{theorem}

\begin{remark}
We shall apply Proposition~\ref{ExplicitPiHol-Poly}, the main objective of the previous subsection,  to the formulas in the theorem above. This application yields convenient formulas for the Fourier expansions of important modular forms (for example, see (\ref{ExplicitPiHol-Poly})) constructed in the next section.
\end{remark}

\section{Bounds for weighted sums of class numbers}\label{WeightedSumsOfClass}

Here we assemble the required asymptotics for the weighted sums of class numbers that lead to the proofs
of Theorems~\ref{Theorem2F1} and ~\ref{Theorem3F2}. The proofs of these asymptotics rely on standard bounds for class numbers and coefficients of cusp forms, and the results of Section~\ref{HarmonicMaassForms} on the holomorphic projection of those nonholomorphic modular forms arising from the Rankin-Cohen bracket of Zagier's $\mathcal{H}(\tau)$ function with certain univariate theta functions.

\subsection{Some Standard Bounds}
Here we recall some simple class numbers bounds, and the celebrated theorem of Deligne which bounds the coefficients of integer weight cusp forms.

\begin{lemma}\label{StandardBounds}The following are true.

\noindent
(1) If $-D<0$ is a discriminant, then we have
$H^\ast(D)\leq \sqrt{D}(\log D+2)/\pi.$

\noindent
(2) For fixed positive integers $r$ and $m,$ as the primes $p\rightarrow +\infty,$ we have
$$
\sum\limits_{s\in \Omega_{p^r}}H^\ast\left(\frac{4p^r-s^2}{4}\right)s^m=o_{m,r}(p^{r(m/2+1)}),
$$
where $\Omega_{p^r}:=\{ s\in [-2\sqrt{p^r}, 2\sqrt{p^r}] \ : \  p\mid s \ {\text {\rm and}}\ s\equiv p^r+1\pmod 4\}.$
\end{lemma}
\begin{proof} Claim (1) is Lemma 2.2 of \cite{griffin-ono-tsai}.
To prove (2), we note that at most $2p^{r/2-1}$ nonzero integers $s$ such that $s^2\leq 4p^r$ and $p|s.$
Therefore, we have the following trivial bound
$$
\sum\limits_{s\in \Omega_{p^r}}H^\ast\left(\frac{4p^r-s^2}{4}\right)s^m\leq 2p^{r/2-1}(2p^{r/2})^m\cdot\max\left\{H^\ast\left(\frac{4p^r-s^2}{4}\right)\right\}.
$$
Claim (2) follows immediately now from (1).
\end{proof}

The following celebrated theorem of Deligne, which bounds the coefficients of integer weight cusp forms, shall also play a key role in our subsequent work. 

\begin{theorem}[Remark 9.3.15 of \cite{stromberg}]\label{Deligne}
If $f=\sum\limits_{n\geq 1} a(n)q_\tau^n$ is a cusp form of integer weight $k$ on a congruence subgroup, then for all $\varepsilon>0$ we have
$a(n)=O_\varepsilon(n^{(k-1)/2+\varepsilon}).$
\end{theorem}

\subsection{Weighted sums of class numbers required for Theorem~\ref{Theorem2F1}}\label{ClassSum2F1}

We now derive the asymptotic formulas which are crucial for the  proof of Theorem~\ref{Theorem2F1}.

\subsubsection{Even moments}\label{EvenMoments2F1Asymptotics}

We begin by recalling the famous classical result of Eichler.
\begin{theorem}[Eichler \cite{Eichler, Eichler2}]\label{EichlerRelation} If $N$ is odd, then
$$
\sum\limits_{-\sqrt{N}\leq s\leq \sqrt{N}}H^*(N-s^2)=-\lambda_1(N)+\frac{1}{3}\sigma_1(N),
$$
where $\sigma_1(N):=\sum\limits_{d|N}d,$ and
$\lambda_{1}(N):=\frac{1}{2}\sum\limits_{d| N}\min(d,\frac{N}{d}).$  
\end{theorem}

From Eichler's identity, if $q=p^r$, where $p$ is an odd prime, then we find that
$$
3\sum\limits_{-\sqrt{q}\leq s\leq \sqrt{q}}H^*(q-s^2)=q+o_{r}(q)
$$
This conclusion is the $n=0$ case of the following general family of asymptotics.

\begin{lemma}\label{AsymptoticsEvenLegendre}
If $n$ is a nonnegative integer, then
$$
3\sum\limits_{s\equiv q+1\pmod 4} H^\ast\left(\frac{4q-s^2}{4}\right)s^{2n}=\frac{(2n)!}{n!(n+1)!}\cdot q^{n+1}+o_{n}(q^{n+1}).
$$
\end{lemma}

\begin{proof}
Since  $H^\ast(D)=0$ for $D\equiv 1, 2\pmod 4,$ we have 
$$
\sum\limits_{s\equiv q+1\pmod 4} H^\ast\left(\frac{4q-s^2}{4}\right)s^{2n}=2^{2n}\sum\limits_s s^{2n}H^\ast(q-s^2).
$$ 
Mertens recently proved Cohen's Conjecture (see Conjecture I.2.1 of \cite{mertens} and \cite{cohen}) which constructs an infinite sequence
of cusp forms from Hurwitz class numbers.
Namely, if $n$ is a positive integer, then he proves (see Theorem 1 of \cite{mertensRMS})  that  the coefficient of $X^{2n}$ in
$$
\sum\limits_{l\text{ odd}}\left[\sum\limits_{s\in\Z}\frac{H^\ast(l-s^2)}{1-2sX+lX^2}+\sum\limits_{k=0}^\infty \lambda_{2k+1}(l)X^{2k}\right]q_\tau^l
$$
is a cusp form of weight $2n+2$ on $\Gamma_0(4),$
where  $\lambda_{2k+1}(l):=\frac{1}{2}\sum\limits_{d|l}\min(d,\frac{l}{d})^{2k+1}.$ 
On the other hand, Lemma 7.5 of \cite{cohen} establishes that the coefficient of $X^{2n}$ is  the Fourier series
$$
\sum\limits_{l\text{ odd}} q_\tau^l\left[\sum\limits_{0\leq t\leq n} (-1)^t\frac{(2n-t)!l^t}{t!(2n-2t)!}\sum\limits_s H^\ast(l-s^2)(2s)^{n-2t}+\lambda_{2n+1}(l)\right].
$$ 

We now prove the lemma by mathematical induction on $n$.
Thanks to Eichler's Theorem~\ref{EichlerRelation}, the claim holds for $n=0.$
Now, suppose that the lemma is true for $n'<n.$
It is clear that $\lambda_{2n+1}(q)=O(q^{n+3/4})=o(q^{n+1})$ as $q\to\infty.$
Therefore, Deligne's Theorem~\ref{Deligne} implies that
\begin{equation}
\sum\limits_{1\leq t\leq n}(-1)^t\frac{(2n-t)!q^t}{t!(2n-2t)!}\sum\limits_s H^\ast(q-s^2)(2s)^{2n-2t}+\sum\limits_s H^\ast(q-s^2)(2s)^{2n}=o_{n}(q^{n+1}).
\end{equation}
By the induction hypothesis, replacing 
$\sum\limits_s H^\ast(q-s^2)(2s)^{2n-2t}$ by $\frac{(2n-2t)!}{3(n-t)!(n-t+1)!}q^{n-t+1}$ contributes $o_n(q^{n-t+1}).$
Therefore, we have
\begin{align}
&3\sum\limits_s H^\ast(q-s^2)(2s)^{2n}\notag=-\sum\limits_{1\leq t\leq n}(-1)^t\frac{(2n-t)!}{t!(n-t)!(n+1-t)!}\cdot q^{n+1}+o_{n}(q^{n+1})\notag
\\
&\ \ \ \ \ \ \ \ \ \ \ \ \ \ \ =\frac{(2n)!}{n!(n+1)!}q^{n+1}-\sum\limits_{0\leq t\leq n}(-1)^t\frac{(2n-t)!}{t!(n-t)!(n+1-t)!}\cdot q^{n+1}+o_{n}(q^{n+1}).\notag
\end{align}
Cohen computed $(1-t)^{l+1}\cdot(1-t)^{-l-1}$ in two ways, and proved 
(see p. 284 of \cite{cohen}) that 
$$
\sum\limits_{0\leq t\leq n}(-1)^t\frac{(2n-t)!}{t!(n-t)!(n+1-t)!}=0,
$$
thereby completing the proof.

\end{proof}

\subsubsection{Odd moments}

The following lemma provides an asymptotic formula for a modified version of the weighted sum of Hurwitz class numbers which appears in Proposition~\ref{Moments2F1} (4).

\begin{lemma}\label{AsymptoticsOddLegendre}
If $m$ is a positive odd integer, then the following are true.

\noindent
(1) As $q\to\infty$ with $q\equiv 1\pmod 4,$ we have 
$$
\sum\limits_{s\equiv q+1\pmod 8}H^\ast\left(\frac{4q-s^2}{4}\right)s^m=o_m(q^{m/2+1}).
$$

\noindent
(2) As $q\to\infty$ with $q\equiv 1\pmod 4,$ we have
$$
\sum\limits_{s\equiv q+1\pmod{16}} H^\ast\left(\frac{4q-s^2}{16}\right)s^m=o_m(q^{m/2+1}).
$$
\end{lemma}

\begin{proof}
Here we prove case (1) when $q\equiv 1, 5\pmod 8.$ 
The proof of (2) is completely analogous and shall be left to the reader.
To this end, let 
$g(\tau)=\eta(8\tau)^3,$
where $\eta(\tau)=q_\tau^{1/24}\prod\limits_{n=1}^\infty(1-q_\tau^n)$ is the Dedekind eta-function.
It is the weight 3/2 cuspidal theta function on $\Gamma_0(64)$ for the Dirichlet character $\chi_{-4}:=\leg{-4}{\bullet}.$
By Theorem~\ref{ZagierSeries} and Proposition~\ref{HolProjProp}, we have that $\pi_{\text{hol}}(\mathcal{H}\cdot g)$ is a holomorphic modular form of weight $3$ on $\Gamma_0(64)$ and Nebentypus character $\chi_{-4}.$ 
Moreover, since $\mathcal{H}$ has manageable growth at cusps and $g$ is a cusp form, Proposition~\ref{BracketProposition} (2) implies that $\pi_{\text{hol}}(\mathcal{H}\cdot g)$ is a cusp form. Thanks to Proposition~\ref{ExplicitPiHol-Poly},  Lemma~\ref{HolProjExplicit1}, and Theorem~\ref{HolProjExplicit2}, its Fourier expansion is
\begin{equation}\label{firstmoment}
\begin{split}
\sum\limits_{n=1}^\infty\left(\sum\limits_{s\equiv 1\pmod 4} H^\ast(n-s^2)s\right)q_\tau^n &+\frac{1}{4}\sum\limits_{n=1}^\infty\left(\sum\limits_{\substack{t^2-l^2=n \\ t,l\geq 1}}\chi_{-4}(t)\cdot(t-l)^2\right)q_\tau^n\\ 
&+\frac{1}{8}\sum\limits_{n=0}^\infty\chi_{-4}(n)\cdot n^2q_\tau^{n^2}.
\end{split}
\end{equation}
Since  we have
$\sum\limits_{\substack{t^2-l^2=n \\ t,l\geq 1}}(t-l)^2\leq n^{\frac{1}{2}}\cdot d(n),$
where $d(n)$ is the divisor function, it is clear that
$\sum\limits_{\substack{t^2-l^2=n \\ t, l\geq 1}}(t-l)^2=o(n^{3/2}).$
Claim (1) with $m=1$ follows from Theorem~\ref{Deligne} as we have
$$
\sum\limits_{s\equiv 2\pmod 8}H^\ast\left(\frac{4q-s^2}{4}\right)s^m=\sum\limits_{s\equiv 1\pmod 4}H^\ast\left(q-s^2\right)(2s)^m.
$$

We proceed by induction. Suppose that it is true for $m'<m,$. If $\nu=(m-1)/2,$ then
it is easy to verify that
$$
[\mathcal{H}^{+},g]_\nu=\sum\limits_{n=0}^\infty\left(\sum\limits_{j=0}^{\nu}(-1)^j{{\nu+\frac{1}{2}}\choose j}{{\nu+\frac{1}{2}}\choose {\nu-j}} \sum\limits_{s\equiv 1\pmod {4}} s^{2\nu-2j+1}(n-s^2)^j H^\ast(n-s^2)\right)q_\tau^n.
$$
Therefore, as above, we have that
\begin{displaymath}
\begin{split}
&\sum\limits_{n=0}^\infty\left(\sum\limits_{j=0}^{\nu}(-1)^j{{\nu+\frac{1}{2}}\choose j}{{\nu+\frac{1}{2}}\choose {\nu-j}} \sum\limits_{s\equiv 1\pmod {4}} s^{2\nu-2j+1}(n-s^2)^j H^\ast(n-s^2)\right)q_\tau^n \\
&\ \ \ \ \ \ +2^{-2\nu-2}\binom{2\nu+1}{\nu+1}\sum\limits_{n=1}^\infty\left(\sum\limits_{\substack{t^2-l^2=n \\ t,l\geq 1}}\chi(t)(t-l)^{2\nu+2}\right)q_\tau^n 
 +\frac{\kappa}{8\sqrt{\pi}}\sum\limits_{n=0}^\infty \chi(n)n^{2\nu+2}q_\tau^{n^2},
\end{split}
\end{displaymath}
where $\kappa=\kappa(3/2,3/2,\nu)$ is as in Lemma~\ref{HolProjExplicit1}, is a cusp form of weight $2\nu+3$ on $\Gamma_0(64)$ and Nebentypus character $\chi_{-4}.$
The proof follows similarly as in the case $m=1$ with an induction argument for the first sum.
\end{proof}

\subsection{Weighted sums of class numbers required for Theorem~\ref{Theorem3F2}}\label{ClassSum3F2}

We state an asymptotic formula for weighted sums of class numbers which are important to prove Theorem~\ref{Theorem3F2}. For brevity, we sketch the proof as it follows the same arguments from the previous section.

\begin{lemma}\label{AsymptoticsEvenClausen}
If $n$ is a nonnegative integer, then as $q\rightarrow +\infty$ we have
$$
\sum\limits_{s \ even} H^\ast(4q-s^2)s^{2n}=\frac{4}{3}\cdot\frac{(2n)!}{n!(n+1)!}q^{n+1}+o_n(q^{n+1}).
$$
\end{lemma}

\begin{proof}[Sketch of the proof]
First, we write
$$
\sum\limits_{s\ {{even}}}H^\ast(4q-s^2)s^{2n}=2^{2n}\sum\limits_{s}H^\ast(4q-4s^2)s^{2n}.
$$
To adapt the proof of Lemma~\ref{AsymptoticsOddLegendre}, let 
$g(\tau)=\theta(4\tau),$
where $\theta(\tau)=\sum\limits_{n\in\Z}q_\tau^{n^2}=1+2q_{\tau}+2q_{\tau}^4+\dots$ is the usual weight 1/2 Jacobi $\theta$ function. Then $g$ is a modular form of weight $1/2$ on $\Gamma_0(16).$   The proof for $\nu=0$ follows from a simple counting argument for the number of Clausen models. Then the proof by induction follows the same steps as in the proof of Lemma~\ref{AsymptoticsOddLegendre} when modified suitably for the weight 1/2 univariate theta function $g(\tau)=\theta(4\tau).$
\end{proof}

\section{Some distributions}\label{ProbTheory}

To obtain Corollaries~\ref{Distribution2F1} and \ref{Distribution3F2}, we will combine
Theorem~\ref{Theorem2F1} and \ref{Theorem3F2} with the following lemma concerning
 the semicircular and Batman distributions.
To make this precise, we first
let $\p$ denote the set of primes, and fix a positive integer $r$. For each prime $p\in\p,$ we have a function
$$f_p:\F_{p^r}\to[-1,1].$$
In this notation, we have the following important lemma.

\begin{lemma}\label{critical-lemma}
If $r$ is a fixed positive integer, then the following are true.

\noindent
(1)  Suppose that the following asymptotics hold for every positive integer $m:$
$$
\sum\limits_{\lambda\in\F_{p^r}}f_{p}(\lambda)^m=\begin{cases}
   o_{m,r}(1)\ \ \  & \hbox{if $m$ is odd} \\ 
  \frac{(2n)!}{2^{2n}(n+1)!n!}+o_{m,r}(1)\ \ \ & \hbox{if $m=2n$ is even.}
  \end{cases}
$$
If $-1\leq a< b\leq 1,$ then
$$
\lim\limits_{p\to\infty}\frac{|\left\{\lambda\in\F_{p^r} \ : \  f_p(\lambda)\in [a,b]\right\}|}{p^r}=\frac{2}{\pi}\int_a^b \sqrt{1-t^2} dt.
$$

\noindent
(2) Suppose that the following asymptotics hold for every positive integer $m:$
$$
\sum\limits_{\lambda\in\F_{p^r}}f_{p}(\lambda)^m=\begin{cases} o_{m,r}(1) \ \ \ \ \ &{\text { if $m$ is odd}}\\ 
\sum\limits_{i=0}^{m} (-1)^i{m\choose i}\frac{(2i)!}{3^m i!(i+1)!}+o_{m,r}(1)  \ \ \ \ 
&{\text { if $m$ is even.}}
\end{cases}
$$
If $-1\leq a<b\leq 1,$ then
$$
\lim\limits_{p\to\infty}\frac{|\left\{\lambda\in\F_{p^r}\ : \ f_p(\lambda)\in [a,b]\right\}|}{p^r}=\frac{3}{4\pi}\int_a^b f(t)dt,
$$
where
$$
f(t)=\begin{cases}
\frac{3-3|t|}{\sqrt{3+6|t|-9t^2}} &\ \ \ \ {\text if}\  \frac{1}{3}<|t|<1 \\ 
\frac{3+3t}{\sqrt{3-6t-9t^2}}+\frac{3-3t}{\sqrt{3+6t-9t^2}} &\ \ \ \ {\text if}\  |t|<\frac{1}{3} \\ 
0 &\ \ \ \ \text{otherwise}.
\end{cases}
$$
\end{lemma}

\begin{proof}  This result follows via a standard application of the method of moments in probability theory (for example, see Theorems 30.1 and 30.2 of \cite{billings}).
We prove these two cases separately.

\smallskip
\noindent
(1) For each $p\in\p$, consider the probability space $(\Omega_p,\mathcal{F}_p,\mu_p)$, where $\Omega_p=\F_{p^r}, \mathcal{F}_{p}=\mathcal{P}(\F_{p^r}),$ and $\mu_p(A)=|A|/p^r$ for all $A\in\mathcal{F}_{p}.$
For the random variable $X_p=f_{p},$ we have
$$
\lim\limits_{p\to\infty} E(X_p^m)=\begin{cases}
   0\ \ \ & \hbox{if $m$ is odd} \\ 
  \frac{(2n)!}{2^{2n}(n+1)!n!}\ \ \ & \hbox{if $m=2n$ is even.}
\end{cases}
$$
Furthermore, consider the probability space $(\Omega,\mathcal{F},\mu_{ST}),$ where $\Omega:=[-1,1]$, $\mathcal{F}$ is the collection of Lebesgue-measurable subsets of $\Omega$, and $\mu_{ST}$ is the measure 
$\mu_{ST}([a,b]):=\frac{2}{\pi}\int_a^b \sqrt{1-t^2}\ dt$.

For the random variable $X:\Omega\to[-1,1],$ defined by $X(t):=t,$ we have
$$
E(X^m)=\begin{cases}
   0\ \ \ & \hbox{if $m$ is odd} \\ 
  \frac{(2n)!}{2^{2n}(n+1)!n!}\ \ \  & \hbox{if $m=2n$ is even.}
\end{cases}
$$
Since the moment-generating function has a positive radius of convergence, the distribution of $X$ is determined by its moments, and thus $X_p$ converges in distribution to $X.$
Therefore, for $-1\leq a<b\leq 1,$ we have
\begin{displaymath}
\begin{split}
\lim\limits_{p\to\infty}\frac{|\left\{\lambda\in\F_{p^r}\ : \  f_p(\lambda)\in [a,b]\right\}|}{p^r}
&=\lim\limits_{p\to\infty}\mu_p(\{ a\leq X_p\leq b\})
=\mu_{ST}(a \leq X\leq b).
\end{split}
\end{displaymath}

\smallskip
\noindent
(2) The proof of the second case follows {\it mutatis mutandis}. The only change is that
$$
\int_{-3}^3 f(t)t^m dt=
\begin{cases} 0 \ \ \ \ \ &{\text { if $m$ is odd}}\\ 
4\pi\sum\limits_{i=0}^{m} (-1)^i{m\choose i}\frac{(2i)!}{i!(i+1)!}  \ \ \ \ 
&{\text { if $m$ is even,}}
\end{cases}
$$
where $f(t)$ is as in Corollary~\ref{Distribution3F2}.
Since $f$ is odd, it is clear that $\int_{-3}^3 f(t)t^m dt=0$ when $m$ is odd.
By symmetry, when $m$ is even, we have
$$
\int_{-3}^3 f(t)t^m dt=2\int_{-1}^3\sqrt{\frac{3-t}{1+t}}\cdot t^m dt.
$$
By a simple change of variables, we see that 
$$
\int_{-3}^3 f(t)t^m dt=8\int_{0}^1t^{-1/2}(1-t)^{1/2}(1-4t)^m dt.
$$
Thankfully, we can express this integral in terms of the Appell hypergeometric series
$$
F_1(a,b_1,b_2;c;x,y):=\sum\limits_{k,n=0}^\infty\frac{(a)_{k+n}(b_1)_k(b_2)_n}{(c)_{k+n} k!n!}\cdot x^ky^n,
$$
where $c$ is not a nonnegative integer, and  $(r)_n:=\prod\limits_{k=0}^{n-1}(r+k)$ for $n\geq1$, and $(r)_0=1.$
By a formula of Bailey 
(see page 77, (4) of \cite{bailey}), we have
$$
 \int_0^1 t^u\left(1-\frac{t}{a}\right)^v\left(1-\frac{t}{b}\right)^w dt=\frac{1}{u+1}\cdot F_1\left(u+1, -v,-w; u+2;\frac{1}{a},\frac{1}{b}\right).
 $$ 
By letting $u=-\frac{1}{2}, v=\frac{1}{2}, w=m,  a=1,$ and $b=\frac{1}{4},$  we obtain\footnote{This series converges since $-m$ is negative integer.}
\begin{equation}\label{integral}
\int_{-3}^3 f(t)t^m dt=16\cdot F_1\left(\frac{1}{2},-\frac{1}{2},-m;\frac{3}{2};1,4\right).
\end{equation}
To find an exact formula, we will need the  classical hypergeometric series
$$
{_2F_1}\left(\begin{array}{cc}
a & b\\
~&c
\end{array}|z\right):=\sum\limits_{n=0}^\infty\frac{(a)_n(b)_n}{(c)_n}\cdot\frac{z^n}{n!},
$$
where $c$ cannot be a nonnegative integer.
It is straightforward to see that
$$
F_1(a,b_1,b_2;c;1,x)={_2F_1}\left(\begin{array}{cc}
a & b_1\\
~&c
\end{array}|1\right)\cdot {_2F_1}\left(\begin{array}{cc}
a & b_2\\
~&c-b_1
\end{array}|x\right).
$$
Substituting this identity into \eqref{integral}, we obtain
\begin{equation}\label{product}
\int_{-3}^3 f(t)t^m dt=16\cdot {_2F_1}\left(\begin{array}{cc}
\frac{1}{2} & -\frac{1}{2}\\
~&\frac{3}{2}
\end{array}|\ 1\right)\cdot {_2F_1}\left(\begin{array}{cc}
\frac{1}{2} & -m\\
~& 2
\end{array}|\ 4\right).
\end{equation}
Using Gauss' identity (see (1.3) of \cite{bailey}, and \cite{Gauss})
$$
{_2F_1}\left(\begin{array}{cc}
a & b\\
~&c
\end{array}|\ 1 \right)=\frac{\Gamma(c)\Gamma(c-a-b)}{\Gamma(c-a)\Gamma(c-b)},
$$
where $\re(c-a-b)>0,$ we find that
\begin{equation}
\int_{-3}^3 f(t)t^m dt=4\pi\cdot {_2F_1}\left(\begin{array}{cc}
\frac{1}{2} & -m\\
~& 2
\end{array}|\ 4\right)=4\pi\sum_{k=0}^{m}(-1)^k{m\choose k}\frac{(2k)!}{k!(k+1)!}.\notag
\end{equation}
The claim in the proposition follows by an elementary rescaling.
\end{proof}

\section{Proofs of Theorems~\ref{Theorem2F1} and \ref{Theorem3F2} and Corollaries~\ref{Distribution2F1} and \ref{Distribution3F2}}\label{Proofs}

We now prove Theorems~\ref{Theorem2F1} and \ref{Theorem3F2}, and their corollaries.

\begin{proof}[Proof of Theorem~\ref{Theorem2F1}]
Proposition~\ref{Moments2F1} gives a formula for the power moments of the values of the hypergeometric functions ${_2F_1}(\lambda)_q$ in terms of weighted sums of class numbers.
Lemma~\ref{StandardBounds} (2) reduces the statement to Lemmas~\ref{AsymptoticsEvenLegendre} and \ref{AsymptoticsOddLegendre}, thereby concluding the proof.
\end{proof}

\begin{proof}[Proof of Corollary~\ref{Distribution2F1}]
After rescaling, the claim follows from Theorem~\ref{Theorem2F1} and Lemma~\ref{critical-lemma}~(1).
\end{proof}

\begin{proof}[Proof of Theorem~\ref{Theorem3F2}]
By Proposition~\ref{EvenMomentsClausen} and Lemma~\ref{StandardBounds} (2), we have that
$$
\sum\limits_{\lambda\in\F_{p^r}} a_\lambda^\Cl(p^{r})^{2n}=\frac{(2n)!}{n!(n+1)!}\cdot p^{rn+r}+o_n(p^{rn+r}) \ \ \ \ \ 
{\text {\rm and}}\ \ \ \ \ 
\sum\limits_{\lambda\in\F_{p^r}} \phi(-\lambda)a_\lambda^\Cl(p^r)^{2n}=o_n(p^{rn+r}),
$$
for all positive integers $n.$
Since ${_3F_2}(\beta)_q=\phi(-\beta){_3F_2}(1/\beta)_q$ for all $\beta\in\F_q^{\times}$ (see Theorem 4.2 of \cite{greene}), Theorem~\ref{Trace3F2} gives us that 
$$
\phi(\lambda+1)a_\lambda^\Cl(q)^2=\phi(\lambda+1)a_{-\lambda-1}^\Cl(q)^2.
$$
Applying the binomial theorem to the equation in Theorem~\ref{Trace3F2} concludes the proof.
\end{proof}

\begin{proof}[Proof of Corollary~\ref{Distribution3F2}] After rescaling,
the claim follows  from Theorem~\ref{Theorem3F2} and Lemma~\ref{critical-lemma}~(2).
\end{proof}

\end{document}